\documentclass{amsart}
\usepackage{amsmath,amsfonts,amssymb,url}

\newcommand{\size}[1]{\left|#1\right|}
\newcommand{\set}[1]{\left\{#1\right\}}

\newcommand {\Z}{\mathbb Z}

\newcommand{\Gal}{\operatorname{Gal}}
\newcommand{\floor}[1]{\lfloor #1\rfloor}

\newcommand{\FF}{\mathbb{F}}
\newcommand{\GFq}[1]{\mathbb{F}_{#1}}

\newcommand{\GFpn}[2]{\mathbb{F}_{{#1}^{#2}}}

\newcommand{\ZN}[1]{\Z/{#1}\Z}

\newcommand{\Proba}{\varpi}

\newcommand{\refeq}[1]{(\ref{#1})}

\newtheorem{theorem}{Theorem}

\newtheorem{proposition}[theorem]{Proposition}

\newtheorem{lemma}[theorem]{Lemma}

\newenvironment {remarks}
   {\noindent{\it Remarks.}}
   {}
\newcommand{\intervalle}[2]{[#1,#2]}
\newcommand{\ev}{\operatorname{ev}}

\usepackage{hyperref}
\begin{document}

\title[Discrete logarithm computations using Reed-Solomon codes]
{Discrete logarithm computations over finite fields using Reed-Solomon codes}

\author{D.~Augot}
\author{F.~Morain}
\address{INRIA \& LIX \\
\'Ecole polytechnique \\
91128 Palaiseau \\
France}
\email[D.~Augot]{daniel.augot@inria.fr}
\email[F.~Morain]{morain@lix.polytechnique.fr}

\date{}

\begin{abstract}
Cheng and Wan have related the decoding of Reed-Solomon codes to the
computation of discrete logarithms over finite fields, with the aim of
proving the hardness of their decoding. In this work, we
experiment with solving the
discrete logarithm over $\GFpn{q}{h}$ using Reed-Solomon decoding. For
fixed $h$ and $q$ going to infinity, we introduce an algorithm (RSDL)
needing $\tilde{O}(h!\cdot q^2)$ operations over $\GFq{q}$, operating on a
$q\times q$ matrix with $(h+2) q$ non-zero coefficients. We give faster
variants including an incremental version and another one that uses
auxiliary finite fields that need not be subfields of $\GFpn{q}{h}$;
this variant is very practical for
moderate values of $q$ and $h$. We include some numerical results of our
first implementations.
\end{abstract}

\maketitle

\section{Introduction}

The fastest known algorithms for computing discrete logarithms in a
finite field $\GFpn{p}{n}$ all rely on variants of the number field
sieve or the function field sieve. The former is used when $n=1$ (see
\cite{Gordon93b,Schirokauer93,ScWeDe96,Weber96,JoLe03,Schirokauer05,CoSe06})
or $p$ is medium (\cite{JoLeSmVe06} improving on \cite{JoLe06}). The
latter is used for fixed $p$ and $n$ going to infinity (see
\cite{Adleman94,AdHu99,JoLe02,GrHoPaSmVe04} and \cite{Coppersmith84}
for $p=2$ generalized in \cite{Semaev98b}). Some related
computations are concerned with computing discrete logarithms over
tori \cite{GrVe05}. All complexities are $L_{p^n}[c, 1/3]$ where as
usual
$$L_x[c, \alpha] = \exp((c+o(1)) (\log x)^{\alpha} (\log\log
x)^{1-\alpha})$$
as $x$ goes to infinity, $c>0$ and $0\leq \alpha < 1$ being constants.

Traditional index calculus methods over $\GFpn{q}{h} =
\GFq{q}[X]/(Q(X))$ (where $Q$ has degree $h$) look for relations of the
type
\begin{equation}\label{relation}
X^u  \bmod Q(X) =: P(X) = \prod_{i=1}^n p_i(X)^{\alpha_{u, i}},
\end{equation}
where $u$ varies and the $p_i$ belong to a factor base
$\mathcal{B}$ containing
irreducible polynomials in $\GFq{q}$. The polynomial $P(X)$
generically has degree $h-1$, and we must find a way to factor it over
$\mathcal{B}$ using elementary division or sieving techniques.
This {\em collection phase} yields a linear system over $\ZN{(q^h-1)}$
that has to be solved in order to find $\log p_i$. Very often, the
system is sparse and suitable methods are known (structured
elimination, block Lanczos \cite{Montgomery95}, block Wiedemann
\cite{Coppersmith94}).

The second phase ({\em search phase}) requires finding a factorization
of $X^u f(X)$, where we want the discrete logarithm of $f(X)$.

\medskip Our aim in this work is to investigate the use of decoding
Reed-Solomon codes instead of factorization of polynomials in the core
of index calculus methods, following the approach of
\cite{ChWa07,ChengWan04}. Superficially, the code-based algorithm
(called RSDL) replaces relations of the type (\ref{relation}) by
$$X^u \equiv f_A(X):= \prod_{a\in A} (X-a) \bmod Q(X),$$
where $A$ is a subset of a fixed set $S \subset \GFpn{q}{h}$. Such a
relation exists if and only if $X^u\bmod Q(X)$ can be decoded. In case
of successful decoding, the set $A$ (or its complement) is recovered
via factorization. If $S$ has cardinality
$n$, $f_A(X)$ will be of degree $n-h$, which highlights one of the
differences with a classical scheme. 

It will turn out that taking $S = \GFq{q}$, so that $n=q$,
is often the sensible choice to do and therefore our method is
interesting in the case $q$ relatively small. Very much like in
Gaudry's setting \cite{Gaudry09}, we will end up with a method of
complexity $\tilde{O}(h!\cdot q^2)$ operations over $\GFq{q}$, for
fixed $h$ and $q$ tending to infinity. The dependency on $h$ can be
dramatically lowered using a variant based on {\em helper fields},
auxiliary finite fields that need not be subfields of $\GFpn{q}{h}$,
making the variant very practical for moderate $q$ and $h$.

\bigskip
The article starts with a review of the theory and practice of
Reed-Solomon codes (Sections \ref{sct-rs} and \ref{sct-ud}). Section
\ref{sct-dl} comes back to the computation of discrete logarithms. The
analysis will be carried out in Section \ref{sct-analysis}. In Section
\ref{sct-incr}, we give an incremental version of our algorithm,
which is faster in practice. Section \ref{sct-galois} will be
concerned with the use of helper fields and their Galois properties.

\section{Reed-Solomon codes}
\label{sct-rs}

\subsection{Definition and properties}

Let $\FF$ be a field,  and $S = \set{x_1, x_2, \ldots,
x_n} \subset \FF^n$ be fixed, with $x_i\neq x_j$ for $i\neq j$. Define
 the evaluation map:
\[
\ev_S:\begin{array}[t]{ccl}
\FF[X] & \rightarrow & \FF^n\\
 r(X)&\mapsto &(r(x_1),\dots,r(x_n)).
\end{array}
\]
For a given $1\leq k\leq n$, the \emph{Reed-Solomon code} $C_k$ over $F$,
with support $S$ and dimension $k$ is
\[
\set{\ev_S(r(X))|\; r(X)\in\FF[X], \deg r(X) < k}\subset \FF^n,
\]
and the set $S$ is called the \emph{support} of the code, see~\cite{Roth:ITCT2006}
It is a \emph{linear code}
whose elements are called \emph{codewords}. The (Hamming) distance
between $y,z\in\FF$ is
\[
d(y,z)=\size{\set{i\in\intervalle1n|\; y_i\neq z_i}},
\]
and $r(X)$ is at distance $\tau$ from $y=(y_1,\dots,y_n)$ if $d(\ev_S
(r(X)),y)\leq\tau$.  The minimum distance of a general code is the
smallest distance between two different codewords, and the minimum
distance of $C_k$ is known to be equal to $d=n-k+1$.


\subsection{The decoding problem}
Given $C_k$ as above, the decoding problem is: given $y\in\FF^n$, and
$\tau\leq n$, find the codewords $c\in C_k$ within Hamming distance
$\tau$ of $y$. This problem and its complexity depend $\tau$. It is a
NP-complete problem~\cite{Guruswami-Vardy:IEEE_IT2005} for general
finite fields, $n$, $k$ and $\tau$.

For Reed-Solomon codes, this amounts to finding, for any $y\in \FF^n$,
the set:\[
F_{\tau}(y)=\{
r(X)\in\FF[X]|\; \deg f(X)<k,\; d(\ev_S (r(X)),y)\leq \tau
\}.
\]
A given algorithm is said to decode up to $\tau$ if it finds
$F_\tau(y)$ for any $y$.  If $\tau>n-k$ tall solutions can be found by
Lagrange interpolation, and there are $\binom{n}{\tau}q^{k-n-\tau}$ of
them.  On the other hand, when $\tau$ is small enough, we have:
\begin{proposition} (Unique decoding) Let $k$ be fixed and let $\tau
  \leq\floor{\frac{n-k}2}$. Then, for any $y\in\FF^n$, one has
  $\left|F_{\tau}(y)\right|\leq 1$.
\end{proposition}
The decoding problem is a list decoding problem when
$\floor{\frac{n-k}2}<\tau<n-k$, and an a priori combinatorial problem
is to determine how large is the size $\ell$ of $F_{\tau}(y)$, in the
worst case over $y$. Of interest is to find $\tau=\tau(n,k)$ such that
$\ell=\ell(n,k)$ is small and $\tau=\floor{n-\sqrt{(k-1)n}}$ was
achieved, in the breakthrough papers~\cite{Sudan:JOC1997,GuSu1999}. In
the present paper, we consider only unique decoding, since unique
decoding algorithms are simpler and faster.

\section{A fast algorithm for uniquely decoding Reed-Solomon codes}
\label{sct-ud}

Among the many algorithms for decoding Reed-Solomon codes, we have
focused our attention on a variant of the Euclidean algorithm of
\cite{SuKaHiNa75}. This version is due to Gao \cite{Gao:2002}.

Let $y=(y_i) \in \FF^n$ to be decoded, $c=(c_i)\in C_k$ be at distance
$\tau$ from $y$, if it exists, $e=y-c=(e_i)$ the error vector, and 
$E = \{i | e_i \neq 0\}$. The {\em locator polynomial} of $e$ is $v(X)
= \prod_{i \in E} (X-x_i)$, and the decoding problem often reduces to finding
this polynomial. Given a decoding radius $\tau$, the correct behaviour
of a decoding radius is to report failure, when the number of errors
is larger than $\tau$. The following algorithm is correct for
Reed-Solomon codes and $\tau=\floor{\frac{n-k}2}$ (unique decoding).

\subsection{Gao's algorithm}
\label{realgao}

For convenience, we reproduce Algorithm 1a in \cite{Gao:2002}.
We let $(x_i)$ be the support of the code and $(y_i)$ a received word.
Remember that $k = n-d+1$. In our case, we will have $k \simeq n$ and
therefore $d$ small. We denote by PartialEEA($s_0$, $s_1$, $D$) the
algorithm that performs the euclidean algorithm on $(s_0, s_1)$ and
stops when a remainder has degree $<D$. In other words, when this algorithm
terminates, we have computed polynomials $u$ and $v$ such that
$$s_0(X) u(X) + s_1(X) v(X) = g(X)$$
where $g$ is the first remainder that has degree $<D$. We note $P(X)
\div X^k$ for the quotient of $P(X)$ by $X^k$.

\medskip
\noindent {\bf Algorithm 1a}

\noindent
INPUT: $(x_i) \in \FF^n$, $(y_i)\in\FF^n$

\noindent
OUTPUT: the error locator polynomial in case of
successful decoding; \verb+failure+ otherwise.

\medskip
\noindent
Step 0. (Compute $G$) Compute $G(X) = \prod_{i=1}^n (X-x_i)$.

\noindent
Step 1. (Interpolation) Compute $I(X)$ such that $I(x_i) = y_i$ for
all $i$.

\noindent
Step 2. (Partial gcd) Perform PartialEEA with inputs $s_0 = G \div X^{k}$
(of degree $d-1$), $s_1 = I \div X^{k}$ (of degree $\leq d-2$),
$D = (d-1)/2$, at which time
$$u(X) s_0(X) + v(X) s_1(X) = g(X)$$
with $\mathrm{deg}(g) < (d-1)/2$.

\noindent
Step 3. (Division)
divide $G(X)$ by $v(X)$ to get $G(X) = h_1(X) v(X) +
r(X)$. If $r \equiv 0$, return $v(X)$, otherwise return \verb+failure+.

\medskip
The original algorithm adds another step for recovering the codeword
in case of success, but we do not need it for our purposes. In our
case, we will need to factor $v(X)$ to get the error locations.

\medskip
This algorithm has been analyzed in \cite{ChYa08}, where fast
multiplication and gcd algorithms are considered (for the
characteristic $2$ case). We briefly summarize the results. 

Let $M(n)$ be the cost to perform a multiplication of two
polynomials of degree $n$ with coefficients in $\FF$, counted in terms
of operations in $\FF$. Following the algorithms of \cite{GaGe99}, we
find that Step 0 costs $O(M(n))$ and Step 1 costs $O(M(n)\log n)$. 
Step 2 requires computing $G(X) \div X^k$ and $I(X)\div X^k$, which
is just coefficient extraction. PartialEEA requires $O(M(d)\log d)$
operations (note that precise constants are
given in \cite{ChYa08}). Step 3 requires a division of a polynomial of
degree $n$ by one of degree $d \leq n$, which costs $O(M(n))$. The
cost of computing the roots of $v(X)$ will depend on the base field.

\subsection{Improvements}

\subsubsection{Computing $G$}

We may compute the highest terms of
${G \div X^k}$ in time $O(M(n))$ (with a small constant, since the last
step in the product tree will be computing the highest terms).

\subsubsection{Interpolation}

The input to the PartialEEA is
$$s_1(X) = I(X) \div X^k = \sum_{i=1}^n \frac{y_i}{G'(x_i)}
(I_i(X) \div X^k) = \sum_{i=1}^n y_i H_i(X).$$
Note that the $H_i(X)$ are polynomials of degree $\leq
d-2$. We can compute $I_i(X) \div X^k$ by appropriately modifying the
last step of the algorithm using product trees, so as to compute only
the higher order terms of $I_i(X)$. This will not modify the
complexity, but will decrease the constant.

\subsubsection{Reusing data}

If the $x_i$ are fixed (this will be our case), then $G(X)$ can be
precomputed (and $s_0$ deduced from it), as well as $G'(x_i)$. The
polynomials $H_i(X)$ can also be precomputed. Instantiating the
formula for $s_1(X)$ will require $O(n d)$ operations, which is
interesting when $d$ is much smaller than $n$.

\subsection{The special case $S = \GFq{q}$}

\subsubsection{First simplifications}

We can write the cost of our modifications of Algorithm~1a as follows
$$T_G + T_{G \div X^k} + T_{I \div X^k} + T_{PEEA} + T_{v\mid G?},$$
where the notation $T_X$ should be selfexplanatory, the last one
accounting for testing whether $v\mid G$.
Since $G(X) = X^q-X$, we have $T_G = O(1)$ and $T_{G \div X^k} = O(1)$.

Since $S$ may be seen as an arithmetic progression, computing $I$ or
$T_{I \div X^k}$ costs $O(M(n))$ using the techniques of \cite{BoSc05}.
We still have $T_{PEEA} = O(M(d) \log d)$.

\subsubsection{Discarding $v$}

Step 3 amounts to checking whether $v(X)$ factors into linear
factors. The ordinary algorithm requires division of $G(X)$ by
$v(X)$ and in case of success, finding the roots of $v(X)$.

When $q$ is very small, we can find the roots of $v(X)$ in $\GFq{q}$
via successive evaluation of $v(a)$ for $a \in \GFq{q}$ in $O(q)$
additions. This cost would therefore be neglectible. 

For larger $q$, we can use the Cantor-Zassenhaus or Berlekamp
algorithms, starting with the computation of
$X^q \bmod v$ at a cost of $O(M(d) \log q)$. In that case, we can
speed up the factoring process of $v(X)$ when needed (storing
$X^{(q-1)/2}$ for future use when $q$ is odd, etc.). The test $v\mid G$
will cost $O(M(d)\log q)$ for all relations, and in case of success,
will be followed by the total cost to find $(d-1)/2$ roots, that is to
say $O(d M(d) \log q)$ operations (assuming gcd to cost less than
exponentiations).

Also, some product tree of the $v$'s could be contemplated.

We can discard some polynomials $v(X)$ by using Swan's
theorem \cite{Swan62}, via computation of the discriminant of $v(X)$,
for a cost of $O(M(d^2))$ operations.

\subsubsection{Final cost}

In summary, we find
$$T_G = O(1), \quad T_{G \div X^k} = O(1), \quad T_{I \div X^k} = O(M(q)),$$
$$T_{EEA} = O(M(d) \log d), \quad T_{X^q \bmod v} = O(M(d)\log q), 
\quad T_{roots} = O(d M(d)\log q).$$

\section{Discrete logarithms}
\label{sct-dl}

\subsection{Connection with decoding Reed-Solomon codes}
\label{sct-decomp}

Consider $\GFpn{q}{h}$ realized as $\GFq{q}[X]/(Q(X))$, and let $S$ be
any subset of $\GFq{q^h}$, such that $Q(a)\neq 0$ for any $a\in S$,
and $n=\size S$. Let $S_\mu$ the set of subsets of size $\mu$ of $S$.
For $A \in S_\mu$, define
$$f_A(X) = \prod_{a \in A} (X - a).$$
We extend~\cite{ChWa07} in a more general context: the field is not
necessarily finite, and $Q(X)$ is not irreducible. Indeed,
\cite{ChWa07} considered only finite fields, and $S\subset \GFq q$.
\begin{theorem}\label{prop:bouboule}\label{prop:bouboulemore}
  Consider $F/K$ a field extension. Let be fixed a monic $Q(X)\in
  K[X]$, with $\deg Q(X)=h$, and $S\subset F$ have size $n$, such that
  $Q(a)\neq 0$ for all $a\in S$. Let $1\leq\mu\leq n$.  For
  any $f(X)\in K[X]$, $\deg f (X)<\mu$, there exists $A\in S_\mu$, such that
\begin{equation}\label{eq:factorbase}
\prod_{a\in A}(X-a)\equiv f(X) \bmod Q(X)
\end{equation}
if and only if the word
\[
y=\ev_S\left( -{f(X)}/{Q(X)}-X^k\right)
\]
is exactly at distance $n-\mu$ from the Reed-Solomon code $C_k$ of
dimension $k=\mu-h$ and support $S$. All the sets $A$ such
that~\refeq{eq:factorbase} holds can be found by decoding y up to the
radius $n-\mu$.
\end{theorem}
\begin{proof}
  Let $f(X)\in K[X]$ be given, $\deg f(X)<\mu$, and suppose that there
  exists $ A\in S_\mu$, such that $ \prod_{a\in A}(x-a)\equiv
  f(x)\bmod Q(x) $.  Then there exists $ t(X)\in F [X]$, $ \deg
  t(X)=\mu-h=k$, such that $\prod_{a\in A} (x-a)= f(x)+t(x)Q(x)$.  We
  remark that $t(X)$ is monic, and we write $t(X)=X^k+r(X)$, with
  $\deg r(X)<k$. Then
\[
f(X)+(X^k+r(X))Q(X)=\prod_{a\in A}(X-a),
\]
which implies that $r(a)=- {f(a)}/{Q(a)}-a^k$ for $a\in A$. Since
$\size A=\mu$, the word $\ev_S\left( -{f(X)}/{Q(X)}-X^k\right)$ is at
distance $n-\mu$ from $\ev_S(r(X))\in C_k$.

Conversely, if $\ev_S\left( -{f(X)}/{Q(X)}-X^k\right)$ is at distance
exactly $n-\mu$ from $C_k$, there exists $A\in S_\mu$ and $r(X)$ with
$\deg r(X)<k$, such that $r(a)=- {f(a)}/{Q(a)}-a^k$ for $a\in A$. Then
\[
\prod_{a\in A}(X-a)\mid
f(X)+(X^k+r(X))Q(X),
\]
and the equality of the degrees  imply the equality,
\[
\prod_{a\in A}(X-a)=
f(X)+(X^k+r(X))Q(X)
\]
which is a relation of type~\refeq{eq:factorbase}.
\end{proof}
\begin{remarks}
When $\mu$ and $k$
  are such that $n-\mu$ is half the minimum distance of $C_k$, the
  mapping
\[
A\in S_\mu\mapsto \prod_{a\in A}(X-a)\bmod Q(X)
\]
is one-to-one, since we have unique decoding.  Furthermore, when
$S\subset \GFq q$, the number of relations of
type~\refeq{eq:factorbase} is $\binom n \mu$, and the probability of
finding one is thus $\binom n \mu/q^h$ when $f(X)\in\GFq q[X]$ is
picked at random of degree less than $h$ . When some elements of $S$
lie in some extension of $\GFq q$, the probability is more intricate
because of the action of the Galois group, see
Section~\ref{sct-galois}.
\end{remarks}
\subsection{The RSDL algorithm for computing discrete logarithms}

The basic idea is to decompose polynomials using decoding of Reed-Solomon
codes in the inner loop. For ease of presentation, we suppose that $F
= \GFpn{q}{h}$. In Section \ref{sct-galois}, we will present a more
general setting.

\medskip
\noindent
\bigskip
\noindent 
INPUT: 
a) $\GFpn{q}{h} = \GFq{q}[X]/(Q(X))$ where $Q(X)$ is primitive
of degree $h$ over $\GFq{q}$; $\GFpn{q}{h}^* = \langle\omega\rangle$.

b) Two parameters $n$ and $\mu$, describing a Reed-Solomon code $[n,
k=\mu-h, d=n-k+1]$; a subset $S$ of $\GFpn{q}{h}$ of cardinality $n$.

\smallskip
\noindent OUTPUT: the logarithm $\log_\omega(\omega-a)$ for all $a \in S$.

\medskip
\noindent
Step 1. (Randomize) Compute $f(X)=X^u\bmod Q(X)$ for a random $u$.

\medskip
\noindent
Step 2. (Decode) Find $A \in S_\mu$ such that 
\[f_{A}(X)\equiv
  f(X)\bmod Q(X)
  \] using decoding. If this fails then pick another random $u$.

\medskip
\noindent
Step 3. (Recover support) given the error-locator polynomial $v(X)$, compute
$
f_A(X)=G(X) / v(X) = \prod_{a\in A}(X-a)
$;
from which we get the relation
\[
u\equiv\sum_{a\in A}\log(\omega-a) \bmod (q^h-1).
\]
If we have less than $n$ relations, goto step 1.

\medskip
\noindent
Step 4. (Linear algebra) solve the $n\times n$ linear system over
  $\ZN{(q^h-1)}$, which yields the logarithms of $\log(\omega-a)$.

\medskip
From $f_A(X) = G(X) / v(X)$, we can rewrite a relation as
$$X^{u} v(X) \equiv G(X) \bmod Q(X).$$
The corresponding row of the relation matrix will have as many
non-zero coefficients as the degree of $v$, which will be shown to be small.

The search phase (finding individual logarithms) follows the same
scheme.

\subsection{Numerical example}

Consider $\GFpn{13}{3} = \GFq{13}[X]/(X^3+2 X+11)$. We use $(n, k,
\mu) = (13, 7, 10)$, which gives $d = 7$. The support is $S = \{0, 1,
\ldots, 12\}$. The probability of
decomposition is $\approx 0.1302$. We find for instance that
$$X^{15} \equiv X^2 + 9 X + 1 \bmod (Q(X), 13).$$
We have to decode the word:
$$y = \ev_S(-X^{15}/Q(X)-X^7) = (7, 1, 1, 0, 1, 3, 6, 8, 9, 12, 4, 11, 10).$$
The PartialEEA procedure yields
$$u(X)=X^2 + 5 X + 3, \quad v(X)=5 X^3 + 2 X^2 + 3, \quad g(X)=7 X + 6,$$
And the polynomial $v$ factors as $(X-3)(X-8)(X-12)$, so that
$$X^{15} (X-3)(X-8)(X-12) \equiv G(X) \bmod (Q(X), 13).$$

Write $13^3-1 = 2^2 \cdot 3^2 \cdot 61$. Logarithms modulo $2^2$ and
$3^2$ are easy to compute. The matrix $M$ modulo $61$ is given in
\ref{fig1}. Its kernel is generated by 
$$V = \left(\begin{array}{cccccccccccccc}
1&3&52&24&57&9&41&54&42&27&41&35&5&36 \\
\end{array}\right)^t.
$$

Computing the logarithm of $X^2+1$ is done using the relation
$$(X^2+1) X \equiv G(X)/((X (X-2) (X-8))) \bmod Q(X)$$
and therefore
$$\log(X^2 + 1) = 417,$$
using the Chinese remaindering theorem.
(Note that this is a toy example, the logarithm of $X^2+1$ could have
been computed in different ways, factoring it over the factor base directly
for instance.)

\begin{figure}[hbt]
$$M = 
\left(\begin{array}{cccccccccccccc}
 15 & 0 & 0 & 1 & 0 & 0 & 0 & 0 & 1 & 0 & 0 & 0 & 1 & 1 \\
 19 & 0 & 1 & 0 & 0 & 0 & 1 & 0 & 0 & 0 & 0 & 1 & 0 & 1 \\
 33 & 1 & 0 & 0 & 1 & 0 & 0 & 1 & 0 & 0 & 0 & 0 & 0 & 1 \\
 40 & 0 & 0 & 1 & 0 & 0 & 0 & 0 & 1 & 0 & 1 & 0 & 0 & 1 \\
 48 & 0 & 0 & 0 & 1 & 0 & 0 & 0 & 1 & 0 & 0 & 0 & 0 & 1 \\
 51 & 1 & 0 & 0 & 0 & 0 & 0 & 0 & 0 & 1 & 0 & 0 & 1 & 1 \\
 0 & 0 & 0 & 0 & 0 & 0 & 0 & 1 & 0 & 1 & 0 & 0 & 1 & 1 \\
 8 & 0 & 0 & 0 & 0 & 1 & 0 & 0 & 1 & 1 & 0 & 0 & 0 & 1 \\
 15 & 1 & 0 & 0 & 0 & 0 & 0 & 0 & 0 & 1 & 1 & 0 & 0 & 1 \\
 25 & 0 & 0 & 0 & 0 & 0 & 0 & 1 & 0 & 1 & 1 & 0 & 0 & 1 \\
 31 & 0 & 0 & 1 & 1 & 0 & 0 & 0 & 0 & 0 & 0 & 1 & 0 & 1 \\
 36 & 0 & 0 & 0 & 0 & 1 & 1 & 0 & 0 & 0 & 0 & 0 & 0 & 1 \\
 48 & 1 & 0 & 0 & 0 & 0 & 0 & 1 & 1 & 0 & 0 & 0 & 0 & 1 \\
 14 & 0 & 1 & 1 & 1 & 0 & 0 & 0 & 0 & 0 & 0 & 0 & 0 & 1 \\
 16 & 0 & 0 & 1 & 0 & 0 & 0 & 0 & 0 & 0 & 1 & 0 & 1 & 1 \\
 17 & 1 & 0 & 0 & 1 & 1 & 0 & 0 & 0 & 0 & 0 & 0 & 0 & 1 \\
 22 & 0 & 0 & 0 & 1 & 0 & 1 & 0 & 0 & 1 & 0 & 0 & 0 & 1 \\
 24 & 0 & 0 & 0 & 1 & 0 & 0 & 0 & 0 & 0 & 0 & 0 & 1 & 1 \\
 27 & 0 & 0 & 0 & 1 & 1 & 0 & 1 & 0 & 0 & 0 & 0 & 0 & 1 \\
\end{array}\right)
$$
\caption{Matrix modulo 61 for the example. \label{fig1}}
\end{figure}


\subsection{Algorithmic remarks}

The inner loop of the algorithm is the computation of
$$y=\ev\left(-\frac{f(X)}{Q(X)}-X^k\right)\in\GFq{q}^n,$$
followed by the interpolation of $y$ on the support, to get $I(X)$. We
can greatly simplify the work by noting that
\begin{lemma}\label{fdallem}
Let $\tilde{Q}(X)$ the inverse of $-Q(X)$ modulo $G(X)$. Then
$$I(X) = (f(X) \tilde{Q}(X) \bmod G(X)) - X^k.$$
\end{lemma}
Since $\tilde{Q}(X)$ is computed only once, the cost of evaluating
$I(X)$ is just $O(M(n))$. From a practical point of view, this is
multiplication by a fixed polynomial modulo a fixed polynomial, a very
well known operation that is very common in computer algebra packages
(in particular NTL).

Moreover, this result shows that we do not need the explicit points
of the support, but rather their minimal polynomial(s). This will be
the key to the incremental version of Section \ref{sct-incr}.

\section{Selecting optimal parameters}
\label{sct-analysis}

\subsection{Unique decoding}


Given $q$ and $h$, we aim to build an optimal $[n,k,n-k+1]_q$
Reed-Solomon code for finding relations~\refeq{eq:factorbase}.  While
Theorem~\ref{prop:bouboule} was used in~\cite{ChWa07} in a negative
way for proving hardness of decoding up to a certain radius, we
consider it in a positive way for solving discrete logarithm problem
using unique decoding. We will consider list decoding in a subsequent
work.
\begin{proposition}
 In the context of Theorem~\ref{prop:bouboule}, to be able to use a unique
 decoding algorithm of the code $C_k$, the parameters should be chosen
 as follows: $\tau=h$, $\mu=n-h$, and $k=n-h$.
\end{proposition}
\begin{proof}
  For  Reed-Solomon codes, unique decoding holds for
  $\tau=\floor{\frac{n-k}2}$. From
  $k=\mu-h=n-\tau-h$, it follows that $\tau=h$.
\end{proof}
It should be noted that $\mu$ and $\tau$ play a symmetrical role.

\subsection{Analyses}

\subsubsection{Set up}

For any integer $s > 0$, we assume that any elementary operation over
$\GFpn{q}{s}$ takes $O(M(\log q^s)) = O(M(s))$ operations over
$\GFq{q}$. In the same vein, an operation over $\ZN{(q^s-1)}$ takes
$M(s)$ operations over $\GFq{q}$. Given that $\tau = h$ and $d =
2h+1$, we will write our complexities in terms of $h$ (which is the
degree of the error-locator polynomial $v(X)$).

The typical analysis involves the probability $\Proba$ to get a
relation (here getting a decoded word). Since we need $n$ relations,
each relation is found after $1/\Proba$ attempts and $c$ operations,
leading to $O(n \frac{1}{\Proba} c)$. Using the decoding approach of
Section \ref{sct-ud}, we see that a more precise count is
$$T_G + T_{G \div X^k} + n \frac{1}{\Proba} (T_{I \div X^k} + T_{EEA}
+ T_{v \mid G?}) + n T_{roots},$$
where we account for reusing $G$ and $G\div X^k$ and perform root
searching of $v$ only in case of success.

The cost of solving a $n\times n$ linear system with $h$
non-zero coefficients per row is $O(h\cdot n^2)$ operations over
$\ZN{(q^h-1)}$, yielding $O(h\cdot n^2\cdot M(h))$ operations over
coefficients of size $\log q$.

We will be fixing $h$ and letting $q$ go to infinity.

\subsubsection{The ordinary case}

In case $S$ is ordinary, that is $S\subset \GFpn{q}{h}$, all
polynomial operations are to be understood in $\GFpn{q}{h}$. We inject
the complexities of Section \ref{sct-ud}. We
have $T_G = O(M(n))$. The additional cost will be
$$O\left(\left(n \frac{1}{\Proba} \left(M(n) + M(h) \log h +
M(n)\right) + n h M(h) \log q\right) M(h)\right),$$
so that the total cost is
$$O\left(\left(n \frac{1}{\Proba} \left(M(n) + M(h) \log h
\right) + n h M(h) \log q + h \cdot n^2\right) M(h)\right).$$

\subsubsection{The case $S \subset \GFq{q}$}

This implies that $n \leq q$. Moreover, With $\mathcal{Q} = q^h$, we get
$$\Proba = \frac{\binom{n}{\mu}}{\mathcal{Q}} =
\frac{\binom{n}{n-\tau}}{\mathcal{Q}} =
\frac{\binom{n}{\tau}}{\mathcal{Q}} = 
\frac{\binom{n}{h}}{\mathcal{Q}}
\approx \frac{n^h}{h! \cdot \mathcal{Q}},$$
since $h$ is fixed.

Using the fact that most of the operations
are performed in $\GFq{q}$, instead of $\GFpn{q}{h}$, we obtain
$$O\left(n \frac{1}{\Proba} \left(M(n) + M(h) \log h
\right) + n h M(h) \log q\right) + O(h \cdot n^2 M(h)).$$

If $n > \log q$ and $n > h$, this simplifies to
$$O\left(h! (q/n)^h n M(n)\right) + O(h \cdot n^2 M(h)),$$
and the first term always dominates. In order to have something not
too slow, we are driven to taking $n = q$, for a cost of 
$$O(h! \cdot q M(q)) + O(h M(h)\cdot q^2) = O(h! \cdot q M(q)) =
\tilde{O}(q^2).$$
Note that both costs are asymptotically $\tilde{O}(q^2)$, but
with different constants. We cannot balance these two phases easily,
since $h$ and $q$ are given. The only thing we can do is relax the
condition $n \leq q$ using Galois properties (see Section
\ref{sct-galois}).

We call RSDL-FQ the corresponding discrete logarithm algorithm with $S =
\GFq{q}$. One of the advantages of this algorithm is to operate on
$q\times q$ matrices with $2 q + h q$ non-zero coefficients, so that a
typical structured Gaussian elimination process will be very
efficient.
\begin{proposition}
For fixed $h$ and $q$ tending to infinity, the algorithm RSDL-FQ has
running time $O(h! \cdot q M(q))$ and requires storing $O(q)$
elements of size $h \log q$.
\end{proposition}
As a corollary, we see that the interpolation step dominates.
This motivates the following Section, where this cost is decreased.

\subsubsection{Looking for a subexponential behavior}

It is customary to search for areas in the plane $(\log q, h)$
yielding a subexponential behavior for the cost function. The analysis
of the previous section works also in case $h \ll n$. The cost being
$\tilde{O}(h!\cdot q^2)$, we look for $0 \leq \alpha < 1$ such that
$$2 \log q + h \log h \simeq c (\log \mathcal{Q})^{\alpha}
(\log\log\mathcal{Q})^{1-\alpha}.$$
Making the hypothesis that $h \ll \log q$ implies
$$2 \frac{\log \mathcal{Q}}{h} \simeq c (\log
\mathcal{Q})^{\alpha}(\log\log\mathcal{Q})^{1-\alpha},$$
or
$$h = \left(\frac{2 \log \mathcal{Q}}{c \log\log
\mathcal{Q}}\right)^{1-\alpha}.$$
In turn,
$$h \simeq \left(\frac{2 h \log q}{c \log\log
q}\right)^{1-\alpha}, \text{i.e.,} 
\; h \simeq \left(\frac{2 \log q}{c \log\log q}\right)^{1/\alpha-1}.$$
In order to respect the hypothesis $h \ll \log q$, we need $\alpha
\geq 1/2$, and $1/2$ is possible.

\section{The incremental version of the algorithm}
\label{sct-incr}

The idea of this variant is to use $f(X) = X^u$ for increasing values
of $u$, so that we can compute the interpolating polynomial for $u+1$
from that of $u$, noting that $I(X)$ is the real input to
Algorithm 1a.
We first explain how to do this, and then conclude with the
incremental version of our algorithm. We cannot prove that using these
polynomials lead to the same theoretical analysis, but it seems to
work well in practice. Note that the search phase can
benefit from the same idea.

The following result will help us interpolating very rapidly, and
is a rewriting of Lemma \ref{fdallem}.
\begin{proposition}
For $u$ an integer, put $f_u(X) = X^u f_0(X) \equiv c_{h-1} X^{h-1} +
\cdots + c_0 \bmod Q(X)$ and $I_u$ the interpolation polynomial that
satisfies $I_u(x_i) = y_i$ for all $i$. Then
$$I_{u+1} \equiv X I(X) + X^{k+1} - X^k + c_{h-1} \bmod G(X).$$
\end{proposition}

For the convenience of the reader, we give a description of the
incremental operations performed in the relation collection phase. We
claim that we no longer need $y_i$, past the initial evaluation.

\medskip
\noindent
{\bf procedure} StartDecodingAt($f_0$, $(x_i)$)

\medskip
\noindent
0. Precompute $G(X) = \prod_{i=1}^n (X-x_i)$; $\tilde{Q}(X) \equiv
-1/Q(X) \bmod G(X)$; $f = f_0$;

\noindent
1. [first interpolation for $u = 0$:] $I:=\tilde{Q} f\bmod G(X) - X^k$;

\noindent
2. for $u:=1$ to $q^h-2$ do

\ \ \ \ \ $c = $ coefficient of degree $h-1$ of $f$;

\ \ \ \ \ \{ update $I$ \}

\ \ \ \ \ $I = (X I + X^{k+1} - X^k + c) \bmod G$;

\ \ \ \ \ \{ update $f$ to $X^{u+1} \bmod Q(X)$ \}

\ \ \ \	\ $f = X f \bmod Q(X)$;

\ \ \ \ \ if $y$ can be decoded with error-locator polynomial $v(X)$ then

\ \ \ \ \ compute $v(X) = \prod_{i=1}^h (X-e_i)$, set $A = S -
\{e_i\}$, 

\ \ \ \ \ store $(u, \{e_i\})$ corresponding to the relation 

$$X^u \equiv f_A(X) \bmod Q(X) \quad \text{ or } \quad
X^{u} v(X) \equiv G(X) \bmod Q(X).$$

\bigskip Note that the storage is minimal, we need to store $u$
and $h$ elements of $\GFq{q}$ for each relation. The corresponding row
in the matrix modulo $P \mid q^h-1$ will contain one integer modulo
$P$ with $h$ values equal to $1$.

The analysis of this very heuristic version is similar to that of the
original version: we replace some $O(M(n))$ by $O(n)$ in the updating
step for $I$. We find the same cost. From
a practical point of view, we gain a lot, since all operations are now
linear in $n = q$. It is all the more efficient as $G(X) = X^q-X$ and
reduction modulo $G$ costs $O(1)$ operations.

\section{Galois action}
\label{sct-galois}

This section is devoted to the case $S\not\subset \GFq q$, with the
idea of increasing the probability of finding relations by using {\em
  helper fields}. It turns out that $S$ and the relations must be
Galois stable. This is not exactly the same effect as obtained in the
NFS/FFS case (see for instance \cite{JoLe06}), but it results in
smaller matrices.

\subsection{Galois orbits}
We state the property in full generality, for a general field
$K$.
\begin{theorem}\label{th:galois}
  Let $F/K$ be a Galois extension, and $Q(X)\in K[X]$ have
  degree $h$. Let  $\mu>h$ be an
  integer. Let $f(X)\in K[X]$, $\deg f(X)<\mu$, such that there exists
  a unique $A\in S_\mu$, such that
\[
f(X)\equiv \prod_{a\in A}(X-a)\bmod Q(X).
\]
Then $A$ is  stable under $\Gal(F/K)$.
\end{theorem}
\begin{proof}
  We have $ \prod_{a\in A}(X-a)= f(X)+t(X) Q(X)$, for some $t(X)\in
  F[X]$. Then, for any $\sigma\in\Gal(F/K) $, we find:
\[
  \sigma\left(\prod_{a\in A}(X-a)\right)= f(X)+\sigma(t(X))
  Q(X),
\]
where the action of $\sigma$ is naturally extended to polynomials.
Writing  $\sigma(t(X))=u(X)$ for some $u(X)\in F[X]$, and since
$\sigma (f(X))=f(X)$, we get
\[
  \prod_{a\in A}(X-\sigma (a))= f(X)+u(X)  Q(X),
\]
i.e.
\[
\prod_{a\in A}(X-\sigma (a))\equiv f(X)\bmod Q(X).
\]
From the hypothesis of the unicity of $A$, we have $\sigma(A)=A$.
\end{proof}
To use the decoding correspondence, we fix a set $S\subset F$ such
that relations of type~\refeq{eq:factorbase} are sought for sets
$A\subset S$. Then, we can enforce the uniqueness condition by fixing
the parameters $n=\size S$, and $\mu$ to have ``unique decoding'',
i.e.\ $\mu=n-h$.  From the previous Theorem, $S$ must be a union of
orbits under $\Gal(F/K)$. We collect these orbits by their size, i.e.\
\[
S=\bigcup_{i=1}^{e} S_i
\]
where $S_i$ is the union of the orbits of size $i$ contained in $S$,
and $e$ is the maximal orbit size.  Defining $n_i=\size{S_i}$, then $
n= \sum_{i=1}^{e} i n_i$, and $(n_1,\dots,n_{e})$ is a \emph{partition
  of $n$ with restricted summands}. Given $e$ and $n$, we call the set
of such partition set $P_n^{e}$ for short, and its size is
asymptotically\cite{Flajolet-Sedgewick:2009}
\[
\size{P_n^e}\sim\frac 1{e!(e-1)!} n^{e-1}.
\]
Before going further, let us mention that $F/K$ does not need to be a
subfield of $K[X]/Q(X)$, and the following diagrams are perfectly
valid for Theorem~\ref{th:galois} to hold and for all the considerations
in this Section.
\begin{center}
\begin{picture}(110,60)(0,0)
\put (-10,50) {$K[X]/Q(X)$}
\put (55,20) {$S\subset F$}
\put (15,5) {\line(5 ,1 ){50}}
\put (5,10){\line(0 ,1 ){35}}
\put (0,0) {$K$}
\end{picture}
\begin{picture}(110,60)(0,0)
\put (0,50) {$\GFpn {q}{5}$}
\put (55,20) {$S\subset \GFpn q 2$}
\put (15,5) {\line(5 ,1 ){50}}
\put (5,10){\line(0 ,1 ){35}}
\put (0,0) {$\GFq q$}
\end{picture}
\end{center}
\begin{proposition}
  Let $S=\cup_{i=1}^{e} S_i$, with $n_i=\size {S_i}$,  $ n=
  \sum_{i=1}^{e} i n_i$, and suppose that unique decoding holds for
  the parameters $n$ and $\mu$. Then the number of
  relations~\refeq{eq:factorbase} is
\[
N_e(\mu)=\;\sum_{(\mu_1,\dots,\mu_{e})\in P_\mu ^{e}}\
\prod_{i=1}^{e}\binom {n_i}{\mu_i}.
\]
\end{proposition}
\begin{proof}
  Consider a partition $(\mu_1,\dots\mu_{e})$ of $\mu$,
  $\mu=\mu_1+2\mu_2+\dots+e\mu_{e}$, and for each $i$, pick
  $\mu_i$ orbits of size $i$ in $S$, and consider their union
  $O_i$. Then $ \prod_{i=1}^{e}\prod_{a\in O_i}(X-a) $ is a
  decomposition of type ~\refeq{eq:factorbase} of size $\mu$, which is
  Galois stable. Conversely, given a relation $ \prod_{a\in
    A}(X-a)\bmod Q(X)$, with $\size A=\mu$, Theorem~\ref{th:galois}
  indicates that $A$ is Galois stable. For each $i$, letting $O_i$ be
  the set of elements of $A$ with orbit size equal to $i$, and
  $\mu_i=\size{O_i}$, we can write
\[
A=O_1\cup\dots\cup O_{e},
\]
with $\mu=\mu_1+2\mu_2+\dots+e\mu_{e}$, i.e.\ a partition of
$\mu$.  The enumeration formula follows, by considering that there are
$\binom {n_i}{\mu_i}$ ways of choosing $\mu_i$ orbits between $n_i$.
\end{proof}
Then, given $\GFq {q^h}$, in the above situation, the
probability of finding a relation is
\[
\varpi=\frac{N_e(h)}{q^h} 
=\frac{1}{q^h} \left( \sum_{(h_1,\dots,h_{e})\in P_h^{e}}\
\prod_{i=1}^{e}\binom {n_i}{h_i}\right),
\]
from the symmetry of $\mu$ and $\tau=n-\mu$, and using  $\tau=h$.

\subsubsection{Example: $n=q^e$}

We choose $S = \GFpn{q}{e}$, $S_i$ being the
set of all elements in $S$ whose orbits under Galois have size
$i$. Then $n_i
=\frac{1}{i}\sum_{j|i} \mu(j)q^{\frac{i}{j}}
\sim q^i/i$, if $i\mid e$, and zero otherwise.  For $h$
constant and growing $q$, we get a probability of
\begin{align*}
\varpi&=\frac{1}{q^h} \left( \sum_{(h_1,\dots,h_{e})\in P_h ^{e}}\
\prod_{i=1}^{e}\binom {n_i}{h_i}\right)\\
&\sim\frac{1}{q^h} \left( \sum_{(h_1,\dots,h_{e})\in P_h ^{e}}\
\prod_{i=1}^{e}\frac{n_i^{h_i}}{h_i!}\right)\\
&\sim\frac{1}{q^h} \left( \sum_{(h_1,\dots,h_{e})\in P_h ^{e}}\
\prod_{i=1}^{e}\frac{q^{ih_i}}{i^{h_i}h_i!}\right)\\
&=\frac{1}{q^h} \left( \sum_{(h_1,\dots,h_{e})\in P_h ^{e}}\
q^{h_1+2h_2+\dots+e h_e}\prod_{i=1}^{e}\frac{1}{i^{h_i}h_i!}\right)\\
&=\sum_{(h_1,\dots,h_{e})\in P_h ^{e}}\prod_{i=1}^{e}\frac{1}{i^{h_i}h_i!}=c_e(h)
\end{align*}
which does not depend on $q$.  This is much higher than $1/h!$, see
Table~\ref{table:c_e}.
\begin{center}
\begin{figure}
\begin{tabular}{llllllll}
$h$&3&5&7&11&13&31&67\\
\hline
$1/h!$& $0.167$& $0.00833$& $0.000198$& $2.51\ 10^{-8}$& $1.61\ 10^{-10}$& $1.22\ 10^{-34}$& $2.74\ 10^{-95}$\\
\hline
$c_{2}(h)$& $ 0.667$& $ 0.217$& $ 0.0460$& $ 0.000895$& $ 9.13\ 10^{-5}$& $ 4.46\ 10^{-16}$& $ 2.36\ 10^{-45}$\\
\hline
$c_{3}(h)$& & $ 0.175$& $ 0.0697$& $ 0.00356$& $ 0.000783$& $ 1.13\ 10^{-11}$& $ 1.32\ 10^{-31}$\\
\hline
$c_{4}(h)$& & $ 0.467$& $ 0.213$& $ 0.0333$& $ 0.0113$& $ 3.24\ 10^{-8}$& $ 1.03\ 10^{-22}$\\
\hline
$c_{6}(h)$& & & $ 0.407$& $ 0.117$& $ 0.0605$& $ 1.48\ 10^{-5}$& $ 4.11\ 10^{-15}$\\
\hline
$c_{8}(h)$& & & & $ 0.117$& $ 0.0696$& $ 9.79\ 10^{-5}$& $ 5.71\ 10^{-12}$\\
\hline
$c_{9}(h)$& & & & $ 0.0591$& $ 0.0424$& $ 9.06\ 10^{-5}$& $ 1.76\ 10^{-11}$\\
\hline
$c_{12}(h)$& & & & & $ 0.227$& $ 0.00384$& $ 6.67\ 10^{-8}$\\
\hline
\end{tabular}
\caption{The constants $1/h!$, $c_e(h)$, for $e = 2, 3, 4, 6, 8, 9, 12$, and
  $h = 3, 5, 7, 11, 13, 31, 67$.}
\label{table:c_e}
\end{figure}
\end{center}

\subsection{Practice}

Since $S = \GFpn{q}{e}$, we have $G(X) = X^{q^e}-X$. Decoding over $S$
amounts to testing divisibility of $G(X)$ by an error-location
polynomial $v(X)$ whose roots are conjugate under the Frobenius, since
$S$ and the corresponding $A$ are. This means that $v(X)$ is a product
of minimal polynomials of elements of $S$. In other words, we can see
this as decomposing over the basis containing these minimal
polynomials. As a consequence, the matrix of relations will be
smaller, its number of columns being $\sum_i n_i \simeq q^e/e$ instead
of $q^e$.

It is not difficult to adapt the incremental version of our algorithm
to that case.

Assuming all operations take place over
$\GFq q$, we thus have a complexity for the relation step
which is dominated by
$$  C=O\left(n \frac{1}{\varpi} \left(M(n) + M(h) \log
      q\right)\right).
$$
In the case where we take $n = q^e$, this yields
$$C=\tilde{O}\left(\frac{q^{2e}}{c_e(h)}\right).$$
Optimizing the value of $n$ is still on-going work.

\subsection{Numerical example}

Consider $\GFpn{7}{5} = \GFq{7}[X]/(X^5+X+4)$ and a helper field
$\GFpn{7}{2}$. The decomposition base contains 7 polynomials of degree
1 and 21 of degree 2, and its cardinality is $28$. By Table
\ref{table:c_e}, the probability of success is approximately
$0.217$. We find for instance
$$X^{20} (X+3)(X+4)(X+5)(X^2+X+4) \equiv G(X):=X^{49}-X \bmod Q(X).$$

\section{Numerical examples}

\subsection{RSDL-FQ}

We programmed RSDL-FQ in NTL 5.5.2 and made it run on an Intel Xeon
CPU E5520 at 2.27GHz. We took $p = 65537$ and ran the program on
several prime values of $h$ (timings are in seconds rounded to the
nearest integer):
$$\begin{array}{|r||r|r|r|r|r|r|}\hline
 h & \text{update} & \text{EEA} & X^q\bmod v &
 \text{roots} & \log_2 P & \text{linear algebra} \\ \hline
 3 &     67 &    4 &    4 & 3 & 27 &    213 \\
 5 &   1297 &  135 &  104 & 6 & 28 &   3398 \\
 7 &  53007 & 8086 & 5745 & 8 & 97 & 124095\\
\hline
\end{array}$$
Defining polynomials are:
$$W^3 + 6 W - 3, \quad W^5 + W + 3, \quad W^7 + W + 3.$$
For the last column, we indicate the size of the largest prime factor $P$
of $p^h-1$ and the time needed to perform Gaussian inversion on the
system modulo $P$ (using \verb+Magma V2.17-1+ on the same machine).

\subsection{RSDL-HF}
We programmed the collection phase RSDL-HF in NTL 5.5.2 and made it
run on an Intel Xeon CPU E5520 at 2.27GHz, collecting the $v(X)$
unfactored. 

We took $p = 3$ and ran the program on $h=29$, with a helper field of
degree $e=8$ (timings are in seconds rounded to the nearest integer),
and the defining polynomial is $Q:=W^{29} + 2W^4 + 1$. Another example
is $p=101$, $h=11$ and $e=2$.  We also include an example over $\GFq
2$, and extension degree $h=31$, with $e=8$.
$$\begin{array}{|r|r|r||r|r|r|r|r|r|}\hline
p & h &e & \text{update} & \text{EEA} & X^{q^e}\bmod v &  \text{linear algebra} \\ \hline
2&31 & 8 & 9 & 271 & 347 & 0 \\
3&29 &  8 & 2255    &  12456    & 8036  & 2 \\
101&11&2 & 440 &816 & 589& 100\\
\hline
\end{array}$$


\section{Concluding remarks}

Improvements can certainly be made to the present scheme to tackle
more realistic discrete logarithm computations. It seems valuable to
have an approach not using smooth polynomials nor using too much
algebraic factorizations in discrete logarithm computations. This
sheds some light on the relationship between coding theory and
classical problems in algorithmic number theory.

Our investigations on the use of Reed-Solomon decoding for discrete
logarithm computations have just begun. For the time being,
the proposed approach seems to have a worse complexity than its
competitor FFS. Many paths are still to follow. In our setting, the use
of so-called large primes is not clear. In our
case, we can force them by trying to decode $P(X) X^u\bmod Q(X)$ for
fixed $P$ and hoping for several relations, but this does not seem to
decrease the cost of the algorithm.

Some other topics of research include the use of list decoding
algorithms, variants of Reed-Solomon or more general codes. We could
also dream of getting the best of the two worlds, for instance
factoring our $f_A(X)$'s to get more relations. All this is the
subject of on-going work.

\medskip
\noindent {\bf Acknowledgments.} Our thanks go to A.~Bostan,
\'E.~Schost for answering our questions on computer algebra;
M.~Finiasz for helpful discussion, B.~Smith for his careful reading of
the manuscript.

\newcommand{\etalchar}[1]{$^{#1}$}


\end{document}